\documentclass[12pt,a4paper]{article}
\usepackage[utf8]{inputenc}
\usepackage{amsmath}
\usepackage{amsfonts}
\usepackage{amssymb}
\usepackage{amsthm}
\usepackage{hyperref}
\usepackage[left=2cm,right=2cm,top=2cm,bottom=2cm]{geometry}
\usepackage{xcolor}
\usepackage[all,cmtip]{xy}
\usepackage{tikz}

\newcommand{\U}{\mathrm{U}}
\newcommand{\SU}{\mathrm{SU}}
\newcommand{\SO}{\mathrm{SO}}
\newcommand{\su}{\mathfrak{su}}

\newcommand{\ra}{\rightarrow}

\newcommand{\un}{\mathfrak{u}}
\newcommand{\spin}{\mathrm{Spin}}
\newcommand{\SP}{\mathrm{Sp}}

\newtheorem{thm}{Theorem}[section]
\newtheorem{prop}[thm]{Proposition}
\newtheorem{cor}[thm]{Corollary}
\newtheorem{lem}[thm]{Lemma}

\theoremstyle{definition}
\newtheorem{rem}[thm]{Remark}
\newtheorem{defn}[thm]{Definition}

\begin{document}

\title{Symplectic and K\"ahler structures on biquotients}
\author{Oliver Goertsches\footnote{Philipps-Universit\"at Marburg, email: goertsch@mathematik.uni-marburg.de}, Panagiotis Konstantis\footnote{Universit\"at Stuttgart, email: panagiotis.konstantis@mathematik.uni-stuttgart.de}, and Leopold Zoller\footnote{Philipps-Universit\"at Marburg, email: zoller@mathematik.uni-marburg.de}}

\maketitle
\begin{abstract}
  We construct symplectic structures on roughly half of all equal rank biquotients of the form
  $G//T$, where $G$ is a compact simple Lie group and $T$ a torus, and investigate Hamiltonian Lie
  group actions on them. For the Eschenburg flag, this action has similar properties as Tolman's and Woodward's examples of Hamiltonian non-Kähler actions. In addition to the previously known K\"ahler structure on the Eschenburg
  flag, we find another K\"ahler structure on a biquotient ${\mathrm{SU}}(4)//T^3$.
\end{abstract}

\section{Introduction}

Given a Lie group $G$, as well as a subgroup $H\subset G\times G$ acting freely on $G$ by left and right multiplication, the orbit space of this action is called a biquotient of $G$, denoted $G//H$. Starting with the habilitation thesis of Eschenburg \cite{MR758252} biquotients have been of interest for differential geometers, mainly because they admit Riemannian metrics with rare curvature properties: all of them can be equipped with metrics of nonnegative sectional curvature, and some even admit positive sectional curvature. \cite{MR758252, BoeckDiss}

In other types of geometries biquotients feature less prominently. Boyer, Galicki, and Mann constructed $3$-Sasakian structures on some biquotients, in particular on the $7$-dimensional Eschenburg spaces $\SU(3)//S^1$, via reduction of circle actions on the standard $3$-Sasakian sphere. Although Kapovich \cite{KapovitchRHTBiquotients} posed the question which biquotients admit a K\"ahler structure, there seems to be no literature on symplectic or K\"ahler structures on biquotients. To our knowledge, the only known such structure on a (nonhomogeneous) biquotient is implicit in the work of Eschenburg \cite[Theorem 2]{Eschenburg} and Escher and Ziller \cite{MR3205800} where they show that the $6$-dimensional Eschenburg flag $\SU(3)//S_{12}$ is a K\"ahler manifold -- see Section \ref{newexample} below. A general result that gives related information is by Singhof \cite{MR1238411} who showed that any equal rank biquotient $G//T$, where $T$ is a torus, admits a stable almost complex structure.

In this paper we construct symplectic structures on roughly half of all equal rank biquotients of compact simple Lie groups $G//T$, where $T$ is a torus, see Theorem \ref{extensionexistence} below. Given a biquotient $G//T$, and a compact connected subgroup $T\subset H\subset G\times G$, where the ranks of $T, H$ and $G$ coincide, one has a fibration
\[
H/T \longrightarrow G//T \longrightarrow G//H.
\]
In case $G//H$ admits a symplectic structure, one can use a construction of Thurston \cite{MR0402764} to obtain a symplectic structure on $G//T$. Work of Eschenburg \cite{MR758252} helped us to find the correct choices for $H$.

In Section \ref{hamiltonianactions} we observe that our new symplectic structures admit Hamiltonian Lie group actions. As an explicit example, we show that the momentum image of the Eschenburg flag $\SU(3)//S_{12}$ is a Tolman trapezoid \cite{Tolman} and compare it to that of the ordinary full flag manifold $\SU(3)/T^2$. In particular, this sheds new light on Tolman's \cite{Tolman} and Woodward's \cite{Woodward} examples of Hamiltonian non-K\"ahler  actions, see Theorem \ref{thm:tolmanwoodward}. We will compare these examples more closely in the follow-up paper \cite{1903.11684}.

In Section \ref{kaehler} we address the question of existence of K\"ahler structures on equal rank biquotients. Besides the Eschenburg flag mentioned above, we find a K\"ahler structure on one more example, namely on $\SU(4)//S_{12}$. We strongly suspect that all equal rank biquotients admit a K\"ahler structure, and support this conjecture by verifying that for low-dimensional examples the cohomology algebra satisfies the Hard Lefschetz property.

\paragraph*{Acknowledgements.} The authors would like to thank Ben Anthes, Jost-Hinrich Eschenburg, Daniel Greb, and Wolfgang Ziller for sharing their insight on the subject. We also want to express our thanks to Maximilian Schmitt for pointing out to us the relation of our results to Tolman's and Woodward's examples. The third named author is supported by the German Academic Scholarship foundation.
\section{Symplectic structures on biquotients}

Let $G$ be a compact Lie group and $H\subset G\times G$. Then $H$ acts on $G$ via $(h_l,h_r)\cdot g=h_lgh_r^{-1}$. This action has the kernel $H\cap \Delta Z$, where $\Delta Z$ denotes the diagonal of the center of $G$. When the induced action of $H/H\cap\Delta Z$ is free, the quotient of the $H$-action is a smooth manifold. The orbit space will then be called a biquotient and denoted by $G//H$. To simplify the language we will often call the $H$-action free even if only the $H/H\cap\Delta Z$-actions is free.\\
We specifically study the case of a torus $T\subset G\times G$ of maximal dimension acting freely on $G$. By maximal dimension we mean that $\dim T=\mathrm{rank}(G)$. It is our goal to endow the biquotient $G//T$ with a symplectic structure. We first present a construction by which this can be achieved under some additional assumption and later discuss to which of the cases in the classification list of simple biquotients this can actually be applied.\\

\subsection{The construction}\label{construction}
In what follows we assume that there is an equal rank extension $T\subset H\subset G\times G$ such that $H$ is connected, the natural $H$-action on $G$ is free, and the biquotient $G//H$ carries a symplectic form $\omega_B$. We have
\[
G//T\cong G\times_H H/T
\]
with $H$ acting on $H/T$ from the left and the diffeomorphisms explicitly given by
\[
T\cdot g\mapsto [g,eT],\quad [g,hT]\mapsto T\cdot h^{-1}\cdot g.
\]
We have a fiber bundle
\begin{equation}\label{eq:fibrationbiquot}
H/T\ra G//T \ra G//H.
\end{equation}
Any fiber over some orbit $H\cdot g$ can be identified with $H/T$ via the map $hT\mapsto [g,hT]\in G\times_H H/T$. This depends on the choice of a particular $g$ in the orbit $gH$, however two different choices for the identification differ only by multiplication with some element of $H$ on $H/T$. It is well known that a flag manifold $H/T$  carries a symplectic form $\omega_0$ such that the $H$-action on $H/T$ is Hamiltonian. Hence $\omega_0$ induces a symplectic form on every fiber, independent of any choices of representatives.
We want to apply the following

\begin{thm}[{\cite{MR0402764}, see also Theorem 6.1.4 in \cite{MR3674984}}]\label{Thurston}
  Let $F\ra E\ra B$ be a fiber bundle of compact manifolds with structure group $G$ in which $(F,\omega_0)$ and $(B,\omega_B)$ are symplectic. Assume further that $\omega_0$ is invariant under the structure group of the bundle and that $[\omega_0]$ lies in the image of the map $H^2(E)\ra H^2(F)$ (we always assume real coefficients in this article). Then there exists a closed $\omega_F\in\Omega^2(E)$ that restricts to the symplectic form induced on every fiber and for $C>0$ sufficiently large the form $\omega_F+C\pi^*(\omega_B)$ is a symplectic form on $E$.
\end{thm}
The spectral sequence associated to the fiber bundle  \eqref{eq:fibrationbiquot} collapses since the cohomologies of fiber and base are concentrated in even degrees -- which was shown in \cite{MR1160094} and \cite{KapovitchRHTBiquotients}, see Theorem \ref{cohomologyofbiquot} below. Thus, the fiber inclusion induces a surjection $H^*(G//T)\ra H^*(H/T)$. To apply Theorem \ref{Thurston} to a given $G//T$ we thus only need to find an appropriate group $H$ such that $G//H$ admits a symplectic structure.

\subsection{Examples of symplectic biquotients}\label{examplesection}

\paragraph{Simply connected simple groups.}
For a simply connected, simple, compact Lie group $G$, biquotients $G//T$ with $T$ of maximal rank as above have been classified in \cite{MR758252}, up to a certain notion of equivalence, see \cite[p.\ 75]{MR758252} -- roughly, it allows to flip the left and right factor of $T$ and to modify $T$ by automorphisms of $G$; for us it is only important that equivalent biquotients are diffeomorphic. Apart from the usual action of a maximal torus by left multiplication and up to equivalence, the free double-sided actions of tori of maximal rank are given by the following tori in $G\times G$:

\begin{itemize}
\item $G=\SU(n)$ (\cite[Satz 664]{MR758252}): Let $e_1,\ldots,e_n$ be the standard basis of the Lie algebra $\mathfrak{t}$ of the diagonal maximal torus $T$ in $\U(n)$. Let $e=e_1+\ldots+e_n$ denote the generator of the center of $\un(n)$. For an element $x\in\mathfrak{u}(n)$ we denote by $x'$ the projection of $x$ to $\su(n)$ along the complement generated by $e$, and for a pair of elements $(x;y)$ we write $(x;y)':= (x';y')$. Up to equivalence, the only two families of tori inducing a free double-sided action are given by $\mathfrak{s}_{k1},\mathfrak{s}_{k2}$, $1\leq k\leq \lfloor\frac{n}{2}\rfloor$, where
\begin{align*}
\mathfrak{s}_{k1}&=\left\langle (2e_n;e_1+e_n)',(0;e_a-e_1),(0;e_b-e_n),~ 1\leq a\leq k,k+1\leq b \leq n\right\rangle\\
\mathfrak{s}_{k2}&=\left\langle \left(2\sum_{i=1}^k e_i;e_n-e_1+2\sum_{i=1}^k e_i\right)',(0, e_i-e_1),~2\leq i\leq n-1\right\rangle.
\end{align*}
\item $G=\spin(2n),n\geq 4$, and $G=\spin(2n+1),\SP(n),n\geq 2$ (\cite[Satz 75]{MR758252}): For $G=\SP(n)$ we fix the torus $T\subset \U(n)\subset\SP(n)$ from before.
In the other cases we choose $T$ to be the maximal torus of $\spin(2n)$ (resp. $\spin(2n+1)$) which covers the standard diagonal torus in $\U(n)\subset \SO(2n)\subset \SO(2n+1)$.
In any case let $e_1,\ldots,e_n$ be the standard basis of its Lie algebra. Again we set $e=e_1+\ldots+e_n$. Then, up to equivalence, there are two tori in $G\times G$ inducing a free double-sided action. They have the Lie algebras
\begin{align*}
\mathfrak{s}_1&=\langle (e_n;0),(0;e_1-e_n),\ldots,(0;e_{n-1}-e_n)\rangle\\
\mathfrak{s}_2&=\langle (e;0),(0;e_1),\ldots,(0;e_{n-1})\rangle.
\end{align*}

\item On the exceptional Lie groups there are no genuine double-sided free actions of tori of maximal rank (\cite[S\"atze 82,83,84]{MR758252}).
\end{itemize}

For roughly half of these cases we find symplectic structures on the resulting biquotients by giving an extension of the torus as in the previous section. Our constructions are inspired by \cite{MR758252} where the author does not only classify all freely acting tori of maximal rank but also their maximal freely acting extensions. For $\SU(n)$ this extension will satisfy our requirements whereas for the other groups we take a slightly smaller one. All of the constructions are summarized here for the convenience of the reader.

\begin{thm}\label{extensionexistence}
For $G=\SU(n)$ (resp.\ $G=\spin(2n),\SP(n)$) the tori associated to $\mathfrak{s}_{k2}$ (resp.\ $\mathfrak{s}_2$) admit an equal rank extension $H\subset G\times G$ satisfying the properties assumed for construction \ref{construction} above. In particular, the associated biquotients admit a symplectic structure.
\end{thm}

\begin{proof}
Let us begin with $G=\SU(n)$. Let $S'$ be the circle generated by
\[\left(2\sum_{i=1}^k e_i;e_n-e_1+2\sum_{i=1}^k e_i\right)'\]
and set $U:=\{e\}\times \SU(n-1)$ with the second factor embedded in the upper left corner. Then $H:=S'\cdot U$ contains the desired torus. Note that $S'$ normalizes $U$ so we can see $G//H$ as the quotient of the induced $S'$-action on $G//U=\SU(n)/\SU(n-1)$ defined by
\[(s_l,s_r)\cdot A\cdot\SU(n-1)=s_lAs_r^{-1}\cdot\SU(n-1).\]
We can simplify things if, instead of looking at the $S'$-action, we consider the circle $S\subset \U(n)\times\U(n)$ generated by $\left(2\sum_{i=1}^k e_i;e_n-e_1+2\sum_{i=1}^k e_i\right)$. It acts on $\SU(n)/\SU(n-1)$ with the same orbits as $S'$ because the center does not contribute to the double-sided action. Concretely, the elements of $S$ are pairs $(B_z,C_z)$, $z\in S^1$ where $B_z$ is the diagonal matrix with entries $z^2$ from positions $1$ to $k$ and $1$ on the rest of the diagonal and $C_z$ is diagonal with entries $b_{11},\ldots,b_{nn}$ satisfying $b_{11}=z,b_{22}=\ldots=b_{kk}=z^2,b_{nn}=z$ and $b_{ii}=1$ for the remaining entries. The diffeomorphism $\SU(n)/\SU(n-1)\ra S^{2n-1}$ that sends $A\cdot\SU(n-1)$ to the last column $(a_{1n},\ldots,a_{nn})$ of $A$ carries $(B_z,C_z)\cdot A\cdot\SU(n-1)=B_zAC_z^{-1}\cdot\SU(n-1)$ to $(za_{1n},\ldots,za_{kn},\overline{z}a_{(k+1)n},\ldots,\overline{z}a_{nn})$. If we further compose with the diffeomorphism that conjugates coordinates $k+1$ to $n$ the $S$-action is identified with the standard $S^1$-action on $S^{2n-1}$ which yields $G//H\cong\mathbb{C}P^{n-1}$.\\
For $G=\SP(n)$ set $H=S\times \SP(n-1)$, where $S$ is the circle generated by $e$ and $\SP(n-1)$ is embedded in the upper left corner. As before we can regard $G//H$ as the quotient of the $S$-action on $\SP(n)/\SP(n-1)\cong S^{4n-1}$. This time the latter identification directly carries the $S$-action to the standard $S^1$-action on $S^{4n-1}$ so again $G//H\cong\mathbb{C}P^{2n-1}$.\\
The same phenomenon occurs for $G=\spin(2n)$. Define $\overline{H}=S\times\SO(2n-1)$, where $S\subset \SO(2n)$ is the circle generated by $e$ and $\SO(2n-1)$ is embedded in the upper left corner.
Now set $H=(p^2)^{-1}(\overline{H})$, $p$ being the projection $\spin(2n)\ra\SO(2n)$. The group $H$ contains the torus defined by $\mathfrak{s}_2$ and actually $G//H\cong \SO(2n)//\overline{H}$. The latter is the quotient of the $S$-action on $\SO(2n)/\SO(2n-1)$ which is easily exposed to be $\mathbb{C}P^{n-1}$ as before. The remaining questions are whether the $H$ action is indeed free and whether $H$ is connected. Note that $G//H$ is simply connected and does not admit nontrivial coverings. With this in mind, both questions are answered positively by the subsequent discussion of coverings.
\end{proof}

\paragraph*{Coverings.} We have seen that there exists quite a large number of symplectic biquotients of simply connected simple groups. We wish to extend the discussion to the realm of groups with nontrivial fundamental group.\\
Any Lie group $G$ has a universal covering group $\hat{G}$ such that $G=\hat{G}/\Gamma$ for some subgroup $\Gamma$ of the center of $\hat{G}$. We assume $G$ to be simple, so $\Gamma$ is necessarily discrete. Now if $H\subset G\times G$ acts freely on $G$ and $p:\hat{G}\ra G$ is the projection, the group $\hat{H}=(p^2)^{-1}(H)\subset\hat{G}\times\hat{G}$ acts freely on $\hat{G}$ \cite[Satz 35]{MR758252}. Actually $\hat{G}//\hat{H}=G//H$, where $\hat H$ may be disconnected. If we consider instead the action of the identity component $\hat{H}_0$ we obtain a covering $\hat{G}//\hat{H}_0\ra G//H$ by dividing out the $\Gamma$-action.\\
This implies that any free double sided torus action on $G$ is covered by one of the cases in the classification list above. Note however that the converse statement does not necessarily apply as not every free double sided action on $\hat{G}$ induces such an action on $G$.\\
Now assume that for some biquotient $G//T$ we have constructed a symplectic form $\hat{\omega}$ on $\hat{G}//\hat{T}_0$ via an extension $\hat{H}$ of $\hat{T}_0$ satisfying the conditions of construction \ref{construction}. Assume further that the symplectic form $\omega_B$ on $\hat{G}//\hat{H}$ that was used for the construction is invariant under the action of $\Gamma$.
Then by the discussion in Section \ref{hamiltonianactions} (setting $K=\Gamma\times\{e\}$), we see that $\hat{\omega}$ can be chosen $\Gamma$-invariant. In particular, it induces a symplectic form on $G//T$.\\
To conclude, note that all of the cases covered in Theorem \ref{extensionexistence} admit such a choice of $\omega_B$:
\begin{itemize}
\item For $\hat{G}=\SU(n)$, and $\hat{H}$ as in the proof of \ref{extensionexistence}, the identification $\hat{G}//\hat{H}\cong\mathbb{C}P^{n-1}$ sends $\hat{H}\cdot A$ to the last column of $A$ with certain entries conjugated. In particular, under this identification the center of $\SU(n)$ acts on $\mathbb{C}P^{n-1}$ by multiplying entries with certain elements of $S^1$. Thus choosing a symplectic form $\omega_B$ on $\mathbb{C}P^{n-1}$ that is invariant under the standard $\U(n)$-action fulfills the requirements.
\item The center of $\hat{G}=\SP(n)$ is $\{\pm 1\}$ which acts trivially on $\hat{G}/\hat{H}\cong \mathbb{C}P^{2n-1}$.
\item For $\hat{G}=\spin(2n)$ we observe that the action of the center on $\hat{G}/\hat{H}$ factors through the action of the center of $\SO(2n)$ which is $\{\pm1\}$. Again, this acts trivially on $\hat{G}/\hat{H}$.
\end{itemize}

The above discussion is summarized in the following

\begin{prop}
Let $G//T$ be an equal rank biquotient of a simple Lie group $G$, with $T\subset G\times G$ a torus. Then it is covered by one of the biquotients $\hat{G}//\hat{T}_0$ of simply connected Lie groups $\hat{G}$ from the list above. Moreover, if $\hat{G}//\hat{T}_0$ belongs to the cases covered in Theorem \ref{extensionexistence}, then the symplectic structure on $\hat{G}//\hat{T}_0$ descends to $G//T$.
\end{prop}

\section{Hamiltonian actions on biquotients}\label{hamiltonianactions}

\subsection{Invariant symplectic forms}
In our examples we find additional symmetry on $G//T$: Take a closed subgroup $K\subset Z_{G\times G}(H)$. Then there is an induced action of $K$ on $G//T$ and $G//H$ defined by $(k_1,k_2)\cdot T\cdot g=T\cdot k_1gk_2^{-1}$ and $(k_1,k_2)\cdot H\cdot g=H\cdot k_1gk_2^{-1}$. If the $K$-action on $G//H$ is symplectic we can adapt the construction of the symplectic form $\omega$ on $G//T$ such that the $K$-action on $G//T$ is symplectic as well:\\
Suppose we have constructed $\omega=\omega_F+C\pi^*(\omega_B)$ as in Theorem \ref{Thurston} above. The $K$-action on $G//T$ commutes with the projection onto $G//H$ and thus respects fibers. The diagram

\[\xymatrix{
& H/T\ar[dl]\ar[dr] & \\
[g,H/T]\ar[rr]\ar[d]& & [k\cdot g,H/T] \ar[d]\\
G//T\ar[rr]^{\cdot k}& & G//T
}\]commutes, where $H/T\ra [g,H/T],~ hT\mapsto [g,hT]$ identifies the fiber over $H\cdot g$ with $H/T$ (analogous for the fiber over $H\cdot kg$). Since $\omega_F$ pulls back to $\omega_0$ in $H/T$ we see that $k^*\omega_F$ also restricts to the symplectic form on every fiber.
Therefore, if we replace $\omega_F$ by
\[\tilde{\omega}_F=\int_K k^*\omega_F dk\]
it will still restrict to the symplectic forms on the fibers. Potentially replacing $C$ by a bigger constant $\tilde{\omega}_F+C\pi^*(\omega_B)$ is a $K$-invariant symplectic form on $G//T$. Thus $K$ acts in a Hamiltonian fashion by the following
\begin{lem}[Addendum to Theorem 26.1 in \cite{MR770935}]
Let $K$ be a Lie group acting symplectically on a symplectic manifold $(M,\omega)$ with $H^1(M)=0$. Then the action is Hamiltonian.
\end{lem}

Let us have a look which Hamiltonian actions arise on the examples of symplectic biquotients from the previous section. We denote by $S_{ki}$ (resp.\ $S_i$) the torus associated to $\mathfrak{s}_{ki}$ (resp.\ $\mathfrak{s}_i$). For $T$ equal to one of these tori and $H$ the corresponding extensions from the proof of \ref{extensionexistence} we obtain Hamiltonian actions of the following groups on $G//T$:

\begin{itemize}
\item $G=\SU(n)$, $T=S_{k2}$: The centralizer $Z_{G\times G}(H)$ is given by $S(\U(k)\times \U(n-k))\times S$ where $S\subset \SU(n)$ is the circle generated by $e_n'$. Note that the action of the whole centralizer on $G//T$ is not effective. An effective action with the same orbits is obtained by restricting only to the left-hand factor that is $S(\U(k)\times\U(n-k))\times\{1\}$.
\item $G=\SP(n)$, $T=S_2$: We have $Z_{G\times G}(H)=\U(n)\times \SP(1)$ with the second factor in the lower right corner. Again the action of the whole group is not effective for the center of the left factor $\U(n)$ acts trivially.
\item $G=\spin (2n)$, $T=S_2$: Recall from the definition of $H$ that $G//H$ is actually equal to $\SO(2n)//\overline{H}$ and that the action of $Z_{G\times G}(H)$ arises as a pullback of the $Z_{\SO(2n)\times\SO(2n)}(\overline{H})$-action along $Z_{G\times G}(H)\ra Z_{\SO(2n)\times\SO(2n)}(\overline{H})$. Consequently, our actual interest lies in the latter which equals $\U(n)\times\{\pm 1\}$. Note that the right-hand factor as well as the center of $\U(n)$ act trivially on $\SO(2n)//\overline{H}$.
\end{itemize}

\subsection{Isotropy representations}

In the above setting we can determine the isotropy representations of $K$ in terms of base and fiber. This is useful for computing the image of the moment map of the $K$-action and will be applied in Section \ref{EBFMomentmap} below.\\
Let $T\cdot g$ be a fixed point of the $K$-action on $G//T$. Then $H\cdot g$ is a fixed point of the $K$-action on $G//H$ so $K$ acts on the fiber over $H\cdot g$. At this point it is convenient to work with the effective versions of the action so let $\overline{T}=T/T\cap\Delta G$. Then the $\overline{T}$-action on $G$ is really free and for any $k\in K$ there is a unique $\varphi(k)\in \overline{T}$ such that $k\cdot g=\varphi(k)^{-1}\cdot g$. We identify the fiber $F=[g,H/T]$ with $H/T$. Note that there is a well defined left action of $\overline{T}$ on $H/T$ for the central part acts trivially. We have
\[k\cdot[g, h\cdot T]=[k\cdot g,h\cdot T]=[g,\varphi(k)\cdot h\cdot T].\]
Thus $K$ acts by pulling back the $\overline{T}$-action on $H/T$ along the homomorphism $\varphi:K\ra \overline {T}$.\\
Observe that the decomposition $T_{T\cdot g}G//T=T_{T\cdot g}F\oplus V$ is preserved by the isotropy action of $K$, where $V$ is the symplectic complement of $T_{T\cdot g}F$. The latter is $K$-equivariantly isomorphic to $T_{H\cdot g}G//H$ so the isotropy representation at $p$ is isomorphic to the sum of the isotropy representation of $K$ at $gH$ and the pullback along $\varphi$ of the isotropy representation of the $\overline{T}$-action on $H/T$ at $eT$.\\
To define the weights of the representations one uses the symplectic form $\omega$ on $G//T$. Since $\omega$ restricts to $\omega_0$ when identifying $F$ with $H/T$ one can determine the weights coming from $T_{T\cdot g}F$ by computing the weights of the isotropy representation of $T$ on $T_{eT}H/T$ with respect to the orientation given by $\omega_0$ and pulling back along the map $\mathfrak{t}^*\ra\mathfrak{k}^*$ defined by $\varphi$. The weights coming from $V$ agree with the ones of the $K$-action on $T_{H\cdot g}G//H$ using $\omega_B$ for orientation. Note that $\omega|_V$ is not necessarily identified with $\omega_B$ on $T_{H\cdot g}G//H$. However the two are sufficiently close if we choose the constant $C$ big enough.

\subsection{The moment map on the Eschenburg flag}\label{EBFMomentmap}

By the Eschenburg flag we mean the quotient of $\SU(3)$ by the double-sided action associated to $\mathfrak{s}_{12}$ (cf.\ Section \ref{examplesection}).
Consider $G=\SU(3)$ and $T\subset \U(3)\times\U(3)$ the torus with Lie algebra \[\mathfrak{t}=\langle (2,0,0;1,0,1),(0,0,0;1,-1,0)\rangle\] with respect to the standard basis of the standard maximal torus in $\U(3)$. Note that $T$ acts on $SU(3)$ since elements in $T$ have the same determinant in both components. This action has the same orbits as the action associated to $\mathfrak{s}_{12}$. Take $H\subset \U(3)\times\U(3)$ to be the subgroup with Lie algebra $\mathfrak{h}=\langle(2,0,0;1,0,1),(0;\su(2)\oplus 0)\rangle$. We have induced actions on $G//T$ and $G//H$ by left multiplication of $K=S\left(\U(1)\times \U(2)\right)$. We want to determine the image of the moment map associated to the action of the two dimensional diagonal torus $S$ of $K$ by computing the weights of the isotropy representations.\\
We fix the basis $e_1-e_2,e_1-e_3$ of $\mathfrak{s}$ (where the $e_i$ are the standard basis of the Lie algebra of the diagonal torus in $\U(3)$) and also use the corresponding dual basis for $\mathfrak{s}^*$. We identify $G//H$ with $\mathbb{C}P^2$ as in the proof of Theorem \ref{extensionexistence}.
Explicitly, the map from $G//H$ to $\mathbb{C}P^2$ is given by sending $H\cdot g$ to $[g_{13}:\overline{g_{23}}:\overline{g_{33}}]$, where $g_{ij}$ denotes the respective matrix entries of $g$.\\
Using the standard symplectic form on $\mathbb{C}P^2$ the weights of the isotropy representations of the standard (non-effective) $T^3$-action on $\mathbb{C}P^2$ at the fixpoints are
\begin{itemize}
\item $(-1,1,0),~(-1,0,1)$ at $[1:0:0]$
\item $(1,-1,0),~(0,-1,1)$ at $[0:1:0]$
\item $(1,0,-1),~(0,1,-1)$ at $[0:0:1].$
\end{itemize}
The $S$-action on $G//H\cong\mathbb{C}P^2$ can be understood as the pullback of this standard action along the homomorphism $S\ra T^3$ which on the level of Lie algebras is represented by the matrix
\[\begin{pmatrix}
1&1\\1&0\\0&1
\end{pmatrix}\]
using the standard basis for the Lie algebra of $T^3$.
We compute
\[\begin{pmatrix}
1&1&0\\1&0&1
\end{pmatrix}\cdot\begin{pmatrix}
-1&-1&1&0&1&0\\1&0&-1&-1&0&1\\0&1&0&1&-1&-1
\end{pmatrix}=\begin{pmatrix}
0&-1&0&-1&1&1\\-1&0&1&1&0&-1
\end{pmatrix}\]

Thus the weights of the horizontal part of the isotropy representations of the $S$-action on $G//T$ at the fixpoints within the respective fibers are given by
\begin{itemize}
\item $(0,-1),~(-1,0)$ over $[1:0:0]$
\item $(0,1),~(-1,1)$ over $[0:1:0]$
\item $(1,0),~(1,-1)$ over $[0:0:1]$.
\end{itemize}
It remains to find the actual fixed points of the $S$-action on the fixed fibers and compute the weights of the corresponding vertical representations. We fix $(2,0,0;1,0,1),(2,0,0;0,1,1)$ as a basis for $\mathfrak{t}$ and also use the dual basis for $\mathfrak{t}^*$.
The group $H$ consists of all elements of the form
\[\left(\begin{pmatrix}
\det(A)^2& &\\ & 1 &\\ & & 1
\end{pmatrix},\begin{pmatrix}
A & \\
 & \det(A)
\end{pmatrix}\right)\]
with $A\in \U(2)$.
Mapping an element of $H$ displayed as above to $A$ defines an isomorphism $H\cong \U(2)$ that sends the chosen basis of $\mathfrak{t}$ to the standard basis of the Lie algebra of the maximal torus of $U(2)$. We identify $H/T\cong \U(2)/T^2\cong \mathbb{C}P^1$ by projecting onto the second column. The $T$-action from the left on $H/T$ corresponds to the standard (non-effective) $T^2$-action on $\mathbb{C}P^1$. Using the standard symplectic form on $\mathbb{C}P^1$, the $T$-action on $H/T$ has the weights $(-1,1)$ and $(1,-1)$ at the two fixed points $[1:0]$ and $[0:1]$.\\
We want to understand the action on the three fixed fibers. The fixed points of the $S$-action on $G//T$ are represented by matrices $p_1,\ldots, p_6$ which are, in this order, given by

\[\begin{pmatrix}
 1 & &  \\
 & 1& \\
 & & 1\end{pmatrix},\quad
\begin{pmatrix}
 & 1 & \\ & & 1\\ 1& &
\end{pmatrix},\quad
\begin{pmatrix}
 & & 1 \\ 1& & \\ & 1 &
\end{pmatrix},\quad
\begin{pmatrix}
 1& & \\ & & 1\\ & -1&
\end{pmatrix},\quad
\begin{pmatrix}
 & 1&  \\ 1& & \\ & & -1
\end{pmatrix},\quad
\begin{pmatrix}
 & &1 \\ & 1&\\ -1& &
\end{pmatrix},
\]
out of which the pairs $(p_1,p_5)$, $(p_2,p_4)$, and $(p_3,p_6)$ lie in the same fiber over $G//H$. In what follows, we use $p_1$, $p_2$, and $p_3$ for the identification of the respective fibers $[p_i,H/T]\cong H/T$.
We compute
\begin{align*}
\begin{pmatrix}
s& &\\ & \bar{s} &\\ & & 1
\end{pmatrix} p_1&=\begin{pmatrix}
1 & &\\ & 1 &\\ & & 1
\end{pmatrix}p_1\begin{pmatrix}
{s}& &\\ & \bar{s} &\\ & & 1
\end{pmatrix}
\text{ and }
\begin{pmatrix}
s& &\\ & 1 &\\ & & \bar{s}
\end{pmatrix} p_1&=\begin{pmatrix}
s^2 & &\\ & 1 &\\ & & 1
\end{pmatrix}p_1\begin{pmatrix}
\bar{s}& &\\ & 1 &\\ & & \bar{s}
\end{pmatrix}.
\end{align*}
for $s\in S^1$. We deduce that $S$ acts on $[p_1,H/T]\cong H/T$ by pulling back the $T$-action along the homomorphism $\varphi_1:S\ra T$ whose matrix representation on Lie algebras is given by
\[\begin{pmatrix}
1&-1\\-1&0
\end{pmatrix}.\]
Thus in the fiber over $[0:0:1]$ there are 2 fixed points and the weights of the respective isotropy representations are given by the pullbacks
\[\pm\begin{pmatrix}
1&-1\\-1&0
\end{pmatrix}\begin{pmatrix}
1\\-1
\end{pmatrix}=\pm\begin{pmatrix}
2\\-1
\end{pmatrix},\]
where the positive sign is the weight at $p_1$ and the negative sign corresponds to $p_5$.
Analogously we compute
\begin{align*}
\begin{pmatrix}
s& &\\ & \bar{s} &\\ & & 1
\end{pmatrix} p_2&=\begin{pmatrix}
s^2 & &\\ & 1 &\\ & & 1
\end{pmatrix}p_2\begin{pmatrix}
1& &\\ & \bar{s} &\\ & & \bar{s}
\end{pmatrix}
\text{ and }
\begin{pmatrix}
s& &\\ & 1 &\\ & & \bar{s}
\end{pmatrix} p_2&=\begin{pmatrix}
1 & &\\ & 1 &\\ & & 1
\end{pmatrix}p_2\begin{pmatrix}
\bar{s}& &\\ & {s} &\\ & & 1
\end{pmatrix}
\end{align*}
as well as

\begin{align*}
\begin{pmatrix}
s& &\\ & \bar{s} &\\ & & 1
\end{pmatrix} p_3&=\begin{pmatrix}
s^2 & &\\ & 1 &\\ & & 1
\end{pmatrix}p_3\begin{pmatrix}
\bar{s}& &\\ & 1 &\\ & & \bar{s}
\end{pmatrix}\text{ and }
\begin{pmatrix}
s& &\\ & 1 &\\ & & \bar{s}
\end{pmatrix} p_3&=\begin{pmatrix}
s^2 & &\\ & 1 &\\ & & 1
\end{pmatrix}p_3\begin{pmatrix}
{1}& &\\ & \bar{s} &\\ & & \bar{s}
\end{pmatrix}.
\end{align*}
Thus $S$ acts on the respective fibers by pullback along $\varphi_2$ and $\varphi_3$, represented by
\[\begin{pmatrix}
0&-1\\-1&1
\end{pmatrix}\text{ and }\begin{pmatrix}
-1&0\\0&-1
\end{pmatrix}\]
giving rise to the weights
\[\pm\begin{pmatrix}
0&-1\\-1&1
\end{pmatrix}\begin{pmatrix}
1\\-1
\end{pmatrix}=\pm\begin{pmatrix}
1\\-2
\end{pmatrix}
\text{ and }
\pm\begin{pmatrix}
-1&0\\0&-1
\end{pmatrix}\begin{pmatrix}
1\\-1
\end{pmatrix}=\pm\begin{pmatrix}
-1\\1
\end{pmatrix}
,\]
with the positive signs corresponding to $p_2$ and $p_3$ while the negative sign belongs to the weights at $p_4$ and $p_6$.
In total the weights at the six fixed points are given by
\begin{itemize}
\item $(1,0),~(1,-1),~(2,-1)$ at $p_1$
\item $(1,0),~(1,-1),~(-2,1)$ at $p_5$
\item $(0,1),~(-1,1),~(1,-2)$ at $p_2$
\item $(0,1),~(-1,1),~(-1,2)$ at $p_4$
\item $(0,-1),~(-1,0),~(-1,1)$ at $p_3$
\item $(0,-1),~(-1,0),~(1,-1)$ at $p_6$.
\end{itemize}

One checks that for two of the triples above the generated cone is $\mathbb{R}^2$, hence two fixed points get mapped to the interior of the moment image while the rest maps to the vertices. Up to translation, global rescaling, and rescaling of the parallel edges between $(p_3,p_6)$, $(p_2,p_5)$, and $(p_1,p_4)$, the image of the moment map of the $S$-action on $G//T$ has the shape as pictured in the left hand figure below. The dots correspond to the images of the fixed points whereas the (dashed) lines are the image of the the 1-skeleton of the action. Compare this to the right hand figure which shows the image of the moment map of the $S$-action on $G//H$ (up to translation and global rescaling).
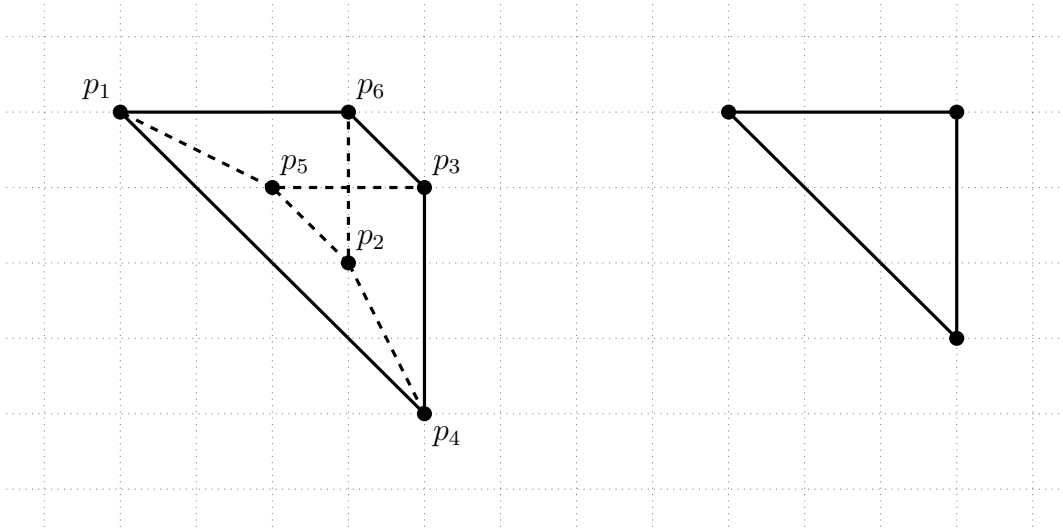
\begin{center}
\begin{tikzpicture}
\draw[step=1, dotted, gray] (-3.5,-4.5) grid (10.5,2.5);

\draw[very thick] (1,1) -- ++(-3,0) -- ++(4,-4) -- ++(0,3) -- ++(-1,1);
\draw[very thick, dashed] (2,-3)--++(-1,+2)--++(0,2)++(0,-2)--++(-1,1)--++(2,0)++(-2,0)--++(-2,1);
  \node at (1,1)[circle,fill,inner sep=2pt]{};
  \node at (1.3,1.3){$p_6$};

  \node at (-2,1)[circle,fill,inner sep=2pt]{};
  \node at (-2.3,1.3){$p_1$};

  \node at (1,-1)[circle,fill,inner sep=2pt]{};
\node at (1.3,-0.7){$p_2$};

  \node at (0,0)[circle,fill,inner sep=2pt]{};
  \node at (0.3,0.3){$p_5$};

  \node at (2,0)[circle,fill,inner sep=2pt]{};
  \node at (2.3,0.3){$p_3$};

  \node at (2,-3)[circle,fill,inner sep=2pt]{};
  \node at (2.3,-3.3){$p_4$};

\draw[very thick] (9,1) -- ++(-3,0) -- ++(3,-3) -- ++(0,3);
\node at (9,1)[circle,fill,inner sep=2pt]{};
\node at (6,1)[circle,fill,inner sep=2pt]{};
\node at (9,-2)[circle,fill,inner sep=2pt]{};

\end{tikzpicture}
\end{center}

Note that at this point we can not expect to obtain the ratios of the lengths of the edges of the left hand figure because these depend on the choice of symplectic form in the following manner: The symplectic form on $G//T$ which we used above is of the form $\omega=\omega_F+C\pi^*(\omega_B)$ for some big enough $C>0$, where $\omega_F$ is closed and $\omega_B$ is the chosen symplectic form on $G//H$. Now if we rescale $\omega$ by considering $C^{-1}\omega_F+\pi^*\omega_B$ we see that for large $C$ the lengths of the edges that are the images of the fixed fibers ($(p_1,p_5)$, $(p_2,p_4)$, and $(p_3,p_6)$) become short and the image of the moment map of $G//T$ approaches that of $G//H$.\\
We point out that the positioning of the inner fixed point images gives an upper bound for the length of the edges coming from the fixed fibers: elongating them would eventually force $p_2$ and $p_5$ to move past each other, which is impossible.

\begin{rem}
Comparing the above picture to the moment image of the $S$-action on the standard flag $\SU(3)/S$, a hexagonal region (cf.\ \cite{MR1414677}), we observe that the latter has much more symmetry. This is due to the fact that the action on $\SU(3)/S$ extends to a Hamiltonian $\SU(3)$-action, which means the Weyl group of $\SU(3)$ acts on the moment image. For the Eschenburg flag however, the $S$-action only extends to an $S(S^1\times \U(2))$-action. The latter has the rather small Weyl group $\mathbb{Z}_2$ which acts in the above picture by reflection at a suitable line in direction $(1,1)$.
\end{rem}

We observe that the momentum image of the $S$-action on the Eschenburg flag is of the same shape as that of Tolman's example \cite{Tolman} of a six-dimensional symplectic $T^2$-manifold  with finitely many fixed points that does not admit an invariant K\"ahler structure: it is a Tolman trapeziod, as it was called in \cite[Section 5.2]{GuilleminSjamaar}. As Tolman's argument for the non-existence of an invariant K\"ahler structure only involves the momentum image, it applies to our situation as well. On the other hand, the Eschenburg flag admits a K\"ahler structure, see \cite{MR3205800} and Section \ref{newexample} below. Note also that Woodward \cite{Woodward} constructed an example of this type admitting a multiplicity-free ${\mathrm{U}}(2)$-action using symplectic surgery, and that our example admits a ${\mathrm{U}}(2)$-action as well. To summarize, our construction shows:
\begin{thm}\label{thm:tolmanwoodward}
There exists a six-dimensional compact, simply-connected manifold $M$ with a ${\mathrm{U}}(2)$-action, such that the restriction to a maximal torus $T^2$ has exactly six fixed points, and which satisfies the following properties:
\begin{enumerate}
\item $M$ admits a K\"ahler structure.
\item There is a symplectic structure on $M$ such that the ${\mathrm{U}}(2)$-action on $M$ is multiplicity-free.
\item $M$ does not admit a $T^2$-invariant K\"ahler structure.
\end{enumerate}
\end{thm}
In \cite{1903.11684} we compare the symplectic Eschenburg flag more closely to Tolman's and Woodward's examples. There, we show all these examples are (non-equivariantly) diffeomorphic -- in particular, Tolman's and Woodward's examples also admit a (noninvariant) K\"ahler structure.

\section{Kähler structures on Biquotients}
\label{kaehler}
\subsection{Flag bundles}
In this section we will remind the reader of some basic
facts concerning flag bundles. More precisely we will explain how one associates a flag bundle to
a complex vector bundle and moreover we will note that if the base manifold is Kähler and if the vector
bundle has a holomorphic structure, then the total space of a flag bundle will be Kähler too.

First, recall the following definition. The ordered set of subspaces $(V_1,\ldots,V_n)$ of a complex
$n$-dimensional vector space $V$ is called a \emph{flag} if $V_1\subset V_2 \subset \ldots V_n=V$,
where $\dim V_i =i$. The set $F(V)$ of all flags is a smooth manifold and can be identified with
$U(n)/T^n$, where $T^n$ is the diagonal $n$-dimensional torus in $U(n)$. More precisely, if
$V = \mathbb C^n$ then
\[
  \Phi \colon \U(n)/T^n \longrightarrow F(\mathbb C^n),\quad
  A \cdot T\mapsto \{ \langle a_1 \rangle, \langle a_1,a_2\rangle ,\ldots,\langle a_1,\ldots,a_{n}\rangle\},
\]
is a bijection where $a_i$ ($1\leq i\leq n$) are the columns of $A$ and $\langle \ldots  \rangle$
means the span over the complex numbers. We say that $F(V)$ is a \emph{(complex) full flag
manifold}.

Second, we call a locally trivial fiber bundle $E \to M$ a \emph{flag bundle} over a manifold $M$ if the
fibers are complex full flag manifolds.

Now if $W \to M$ is a complex vector bundle of rank $n$, then there is a splitting manifold $F(W)$
and a map $\pi \colon F(W) \to M$ such that $\pi^*(W)$ is isomorphic to a Whitney-sum of complex
line bundles over $F(W)$ (cf. \cite[§21]{MR658304}). Moreover $\pi \colon F(W) \to M$ has
the structure of a flag bundle. We will recall briefly this construction, since it will be needed
for this section.

Start with the projectivization $\pi_0 \colon P(W) \to M$ of $W \to M$ which is a locally trivial
bundle with fibers complex projective spaces. There is a line subbundle $S_1$ of $\pi_0^\ast(W)$
such that $S_1$ restricted to the fibers of $\pi_0$ is the tautological bundle. Denote by $Q_1$ the
quotient bundle $\pi_0^\ast(W)/S_1$ and repeat the procedure with $Q_1 \to P(W)=:P_1$ to produce
$\pi_1 \colon P(Q_1) \to P(W)$ and a vector bundle $Q_2 \to P(Q_1)=P_2$. The splitting manifold is
defined to be $F(W):=P_{n-1}$ and the map $\pi = \pi_0 \circ \ldots\circ \pi_{n-1} \colon F(W) \to M$
has the structure of a locally trivial fiber bundle.

Note
thereby, that if $M$ is a point, then $W$ is just a complex vector space and the space $F(W)$ is the
complex full flag manifold defined above. This shows that for a general vector bundle $W \to M$ the map $\pi \colon
F(W) \to M$ has the structure of a flag bundle.

Assume that $M$ is a compact Kähler manifold and $W \to M$ a holomorphic vector bundle. Then we have
a well-known lemma

\begin{lem}[{\cite[Proposition 3.18]{MR2451566}}]\label{ProjectivizationIsKaehler}
  The total space of the projectivization $P(W)$ is Kähler and the projection $\pi_0 \colon P(W)
  \to M$ is holomorphic.
\end{lem}

Thus $\pi_0^\ast(W)$ is again a holomorphic
bundle over $P(W)$. Note that $S_1$ is a holomorphic subbundle of $\pi_0^\ast(W)$, therefore
$Q_1=\pi_0^\ast(W)/S_1$ is again holomorphic. With Lemma \ref{ProjectivizationIsKaehler} we obtain
that $P(Q_1)$ is Kähler and the map $\pi_1 \colon P(Q_1) \to P(W)$ is holomorphic.
Inductively
we finally deduce
\begin{lem}\label{lem:flagkaehler}
  Let $M$ be a compact Kähler manifold and $W \to M$ a holomorphic vector bundle. Then the total space $F(W)$ of the associated flag bundle to $W\to M$ is also Kähler.
\end{lem}

\subsection{A new example}
\label{newexample}

In \cite[Theorem 2]{Eschenburg} and \cite{MR3205800} the Eschenburg flag, that is the biquotient $\SU(n)//S_{k2}$ in the case $n=3$, $k=1$ (with $S_{k2}$ being the torus associated to $\mathfrak{s}_{k2}$ as defined in Section \ref{examplesection}), was written as the projectivization of a holomorphic vector bundle over $\mathbb{C}P^2$. This in particular implies the existence of a Kähler metric on the Eschenburg flag. We wish to extend this strategy to the case $n=4$.

\begin{prop}\label{VB-prop}
For $n\geq 3$, The biquotient $\SU(n)//S_{12}$ is the total space of the flag bundle associated to a complex vector bundle over $\mathbb{C}P^{n-1}$.
\end{prop}

\begin{proof}
Note first that the torus $S\subset \U(n)\times \U(n)$ with Lie algebra \[
\left\langle \left(2e_1;e_1+e_n\right),(0;e_2-e_1),\ldots,(0;e_{n-1}-e_1)\right\rangle\]
acts on $\SU(n)$ in the usual double-sided fashion and that $\SU(n)//S=\SU(n)//S_{12}.$ We consider the equal rank extension $H$ of matrix tuples of the form
\[\left(\begin{pmatrix}
\det(A)^2 & 0\\ 0& I_{n-1}
\end{pmatrix},\begin{pmatrix}
A& 0\\0& \det(A)
\end{pmatrix}\right)
\]
with $A\in \U(n-1)$.  The group $H$ acts freely on $\SU(n)$ and as before we write $\SU(n)//S= \SU(n)\times_H H/S$. Note that sending a tuple of $H$ written as above to the matrix $A$ defines an isomorphism $H\cong \U(n-1)$ under which $S$ maps to the standard diagonal torus $T^{n-1}$. In particular, using the bijection $\Phi$ of the previous section,
\[\SU(n)//S\cong \SU(n)\times_{\U(n-1)} \U(n-1)/T^{n-1}\]
is the flag bundle associated to the complex bundle $\SU(n)\times_{\U(n-1)}\mathbb{C}^{n-1}$. For the identification $\SU(n)//H\cong \mathbb{C}P^{n-1}$ one proceeds as in the proof of Theorem \ref{extensionexistence}.
\end{proof}

\begin{cor}
The biquotients $\SU(3)//S_{12}$ and $\SU(4)//S_{12}$ admit a K\"ahler structure.
\end{cor}

\begin{proof}
  From \cite[Chapter I, §6]{MR2815674} every (topological) vector bundle over $\mathbb{C}P^2$ or $\mathbb{C}P^3$ possesses   a holomorphic structure. Proposition \ref{VB-prop}, combined with Lemma \ref{lem:flagkaehler}, yields the result.
\end{proof}

\begin{rem}
Other biquotients can be regarded from a similar point of view. For instance, for $G=\SP(n)$ and the torus $S_2 = S\times T^{n-1}\subset S\times \SP(n-1)$, the bundle
\[
\SP(n)//S_2 = S\setminus \SP(n) \times_{\SP(n-1)} \SP(n-1)/T^{n-1}\longrightarrow S\setminus \SP(n)/\SP(n-1) \cong {\mathbb{C}} P^{2n-1}
\]
is the bundle of full isotropic flags (with respect to a complex symplectic form) in the complex vector bundle $S\setminus \SP(n) \times_{\SP(n-1)} {\mathbb{C}}^{2n-2} \to {\mathbb{C}}P^{2n-1}$. We expect that one might use this description to find a K\"ahler structure on this space, provided that one can show that this bundle and the complex symplectic form are holomorphic.
\end{rem}

\subsection{The Hard Lefschetz property}
It is natural to ask whether the existence of Kähler structures on biquotients can be excluded due to topological obstructions. The most prominent algebraic topological feature of compact K\"ahler manifolds is the fact that their cohomology algebras satisfy the Hard Lefschetz property (HLP). Recall the following

\begin{defn}
A commutative graded algebra $A$ is said to have the Hard Lefschetz property if there is an element $\omega\in A^2$ such that, for some fixed $n$ and all $k\geq 0$, multiplication with $\omega^k$ induces an isomorphism $A^{n-k}\cong A^{n+k}$. In this case we call $\omega$ a Hard Lefschetz element.
\end{defn}

The study of the HLP and more general Lefschetz properties on certain algebras is an active field of research in commutative algebra. However, results that prove the HLP for whole families of algebras seem to be restricted to very special cases like e.g.\ quotients of polynomial rings by monomial ideals \cite{MR3161738, MR3112920}. The relations appearing in the cohomology rings of biquotients of simple Lie groups (cf.\ Section \ref{examplesection}) are much more complicated and we were not able to prove the HLP for one of the listed families. The purpose of this section is rather to give examples from all families in which the HLP does hold, backing up the authors suspicion that indeed all of the spaces satisfy the HLP and might even be Kähler.

\paragraph*{Cohomology rings of biquotients.}
Before we can give the examples, we need to understand the cohomology rings of the spaces in question. They are easily computable as was shown e.g.\ in \cite{MR1160094} and \cite{MR1238411} using spectral sequences or in \cite{KapovitchRHTBiquotients} via rational homotopy. Let $S(\cdot)$ denote the symmetric algebra. Simplified and tailored to the situation from Section \ref{examplesection} one has

\begin{thm}\label{cohomologyofbiquot}
Let $T\subset G$ a maximal torus and suppose $S\subset T\times T$ induces a free double-sided action on $G$ where $\mathrm{rank}(S)=\mathrm{rank}(G)$. Let $i:\mathfrak{s}\ra \mathfrak{t}\times \mathfrak{t}$ denote the inclusion of Lie algebras. Then
\[H^*(G//S)\cong S(\mathfrak{s}^*)/\left(i^*(\delta_1),\ldots,i^*(\delta_n)\right),\]
where $\delta_i=\sigma_i\otimes 1-1\otimes \sigma_i\in S(\mathfrak{t}^*\times\mathfrak{t}^*)$ for the polynomial generators $\sigma_i$ of the Weyl-invariant polynomials in $S(\mathfrak{t}^*).$
\end{thm}

The presentation of the biquotients in Section \ref{examplesection} comes with bases for $\mathfrak{s}$ and $\mathfrak{t}$ which we fix as bases here. If we denote the corresponding dual basis of $\mathfrak{t}^*$ by $y_1,\ldots y_n$, then the $\sigma_i$ are given by
\begin{itemize}
\item $G=\SU(n)$: The elementary symmetric polynomials of degree $2,\ldots,n$ in the variables $y_1,\ldots,y_n$.
\item $G=\spin(2n+1),\SP(n)$: The elementary symmetric polynomials of degree $1,\ldots,n$ in the variables $y_1^2,\ldots,y_n^2$.
\item $G=\spin(2n)$: The elementary symmetric polynomials of degree $1,\ldots,n-1$ in the variables $y_1^2,\ldots,y_n^2$ as well as the $n$th elementary symmetric polynomial in the variables $y_1,\ldots,y_n$.
\end{itemize}

In particular, it follows that the matching biquotients of $\spin(2n+1)$ and $\SP(n)$ have isomorphic cohomology. Using more general formulations of Theorem \ref{cohomologyofbiquot} one could even see that they have the same rational homotopy type.\\
We are now ready to give explicit examples of elements in $H^2(G//S)$ that are Hard Lefschetz elements. Let $G$ be simple and $S$ be one of the tori $S_k$ (resp.\ $S_{kl}$) from Section \ref{examplesection}. Also let $x_1,\ldots,x_n$ be the dual basis of the basis used for the definition of the Lie algebra $\mathfrak{s}$ of $S$ in \ref{examplesection}. Then, via the isomorphism of Theorem \ref{cohomologyofbiquot}, the element $\omega=\sum_{i=1}^nix_i$ induces an element of $H^2(G//S)$. In the cases where $n\leq 5$, the element $\omega$ is a Hard Lefschetz element. In particular we have the following

\begin{prop}
Let $G$ be simple, $\mathrm{rank}(G)\leq 5$ and $G//S$ a biquotient with $S$ a torus of maximal rank. Then $H^*(G//S)$ satisfies the Hard Lefschetz property.
\end{prop}

\begin{rem}
The bases $x_i$ are chosen noncanonically for every torus $S$. As expected, the sets of coefficients $\alpha_i$ for which $\sum \alpha_i x_i$ induces a Hard Lefschetz element of $H^2(G//S)$ do in general not coincide for different $S$, even for small ranks. It is therefore rather surprising that the coefficients defining $\omega$ above work for all those tori simultaneously.
\end{rem}

The calculations to verify that $\omega$ is indeed a Hard Lefschetz element are too lengthy to be displayed here explicitly. To compensate, we discuss below how the question of whether a certain element of an algebra as above is Hard Lefschetz can be reduced to a problem of ideal membership which is solvable with any standard computer algebra software. Additionally, we demonstrate how to verify the proposition, using the freely available software Macaulay2 \cite{M2}. While we only display the code for the Torus $S_1$ of $G=\SP(n)$ and $G=\spin(2n+1)$, all remaining cases can be checked with only slight modifications.

\paragraph*{Testing for Hard Lefschetz elements.}
The cohomology algebras of the biquotients above are of the form $H=\mathbb{R}[x_1,\ldots, x_n]/(f_1,\ldots,f_n)$ where the $x_i$ are of degree 2 and the $f_i$ are homogeneous polynomials (in fact the $f_i$ form a regular sequence). We have the additional information that $H$ satisfies Poincaré duality with fundamental class in even degree $2m$.
For dimensional reasons, an element $\omega\in H^2$ is Hard Lefschetz if and only if multiplication with $\omega^k$ is surjective onto the degree $m+k$ component of $H$.

Let $p\in \mathbb{R}[x_1,\ldots,x_n]$ be a representative of $\omega$. The cokernel of multiplication by $\omega^k$ is just $\mathbb{R}[x_1,\ldots,x_n]/(f_1,\ldots,f_n,p^k)$.
Thus $\omega$ is Hard Lefschetz if and only if for $k=1,\ldots,m$ the degree $m+k$ component of $\mathbb{R}[x_1,\ldots,x_n]/(f_1,\ldots,f_n,p^k)$ is trivial. This is of course equivalent to the question of whether all degree $m+k$ monomials lie in the ideal $(f_1,\ldots,f_n,p^k)$.

We present in the following lines a Macaulay2 algorithm for testing the Hard Lefschetz
property.
\begin{verbatim}
-- Sp(n) or Spin(2n+1)
n = 3; -- dimension of a maximal torus, change at will

R=QQ[x_1..x_n,Degrees=> for i from 1 to n list 2];
S=QQ[y_1..y_(2*n)]

m=n^2; -- Dimension of biquotient / 2

-- the x_0-th elementary symmetric polynomial evaluated at x_1,...,x_n.
sigma = x -> (
    if (not class x === Sequence) or #x < 2 or x_0 < 0 then (
	error "Wrong input data";
	return 0;
    );
  i := x_0;
  n := #x-1;
  args := toList drop(x,1);
  args' := drop(args,-1);
  if i > n then return 0;
  if i == 0 then return 1;
  if i == 1 then return sum args;
  return x_n*sigma(toSequence prepend(i-1,args'))
      +sigma(toSequence prepend(i,args'));
);

-- create lists with squared generators
L_1 = for i from 1 to n list y_i^2;
L_2 = for i from n+1 to 2*n list y_i^2;
for i from 1 to n do (
  d_i=sigma(prepend(i, toSequence L_1)) - sigma(prepend(i, toSequence L_2));
	);

-- Linear transformation between the symmetric algebras
-- (see Theoerem 4.7) for S_1

M = for i from 1 to n-1 list 0;
for i from 1  to n do M=append(M,x_i);
M=append(M,sum for i from 2 to n list -x_i);
g=map(R,S, M);

--create polynomial relations of the cohomology
I = for i from 1 to n list g(d_i);
-- define omega
w = sum for i from 1 to n list i*x_i;
-- main algorithm
for k from 1 to m do (
  j=basis(m+k,R/ideal(join(I,{w^k})));
  if j != 0 then break;
	);

if  j == 0 then print "Hard Lefschetz!" else print "Not Hard Lefschetz!";

quit()
\end{verbatim}
\bibliographystyle{acm}
\bibliography{BiquotSymplectic}
\end{document}